\documentclass{article}
\usepackage{authblk}
\usepackage[margin = 3.75cm]{geometry}      

\usepackage{amsthm, amssymb, latexsym, graphics}
\usepackage{amsmath,amsfonts}
\usepackage{hyperref}

\usepackage{color}

\newtheorem{theorem}{Theorem}[section]
\newtheorem{lemma}[theorem]{Lemma}

\newtheorem{proposition}[theorem]{Proposition}

\theoremstyle{definition}

\theoremstyle{definition}

\newcommand\R{{\mathbb R}}

\newcommand{\authorrunning}[1]{}
\newcommand{\institute}[1]{}

\author[1]{Fabian Bleitner}
\author[2]{Camilla Nobili}

\affil[1]{\small{Department of Mathematics, University of Hamburg, Germany}}
\affil[2]{\small{Department of Mathematics, University of Surrey, United Kingdom}}

\date{}

\begin{document}

\title{Lower Bounds for the Advection-Hyperdiffusion Equation}
\authorrunning{F.~Bleitner, C.~Nobili}
\institute{Fabian Bleitner
\at Department of Mathematics, Universit\"at Hamburg, Germany\\ 
\email{fabian.bleitner@uni-hamburg.de}
\and 
Camilla Nobili
\at Department of Mathematics, University of Surrey, United Kingdom\\ 
\email{c.nobili@surrey.ac.uk}}

\maketitle

\abstract{
	Motivated by \cite{nobiliPottel2021lowerBoundsOnMixingNormsForTheAdvectinDiffusionEquationInRd}, we study the advection-hyperdiffusion equation in the whole space in two and three dimensions with the goal of understanding the decay in time of the $H^{-1}$- and $L^2$-norm of the solutions. We view the advection term as a perturbation of the hyperdiffusion equation and employ the Fourier-splitting method first introduced by Schonbek in \cite{schonbek1980DecayOfSolutionToParabolicConservationLaws} for scalar parabolic equations and later generalized to a broader class of equations including Navier-Stokes equations and magneto-hydrodynamic systems. This approach consists of decomposing the Fourier space along a sphere with radius decreasing in time. Combining the Fourier-splitting method with classical PDE techniques applied to the hyperdiffusion equation we find a lower bound for the $H^{-1}$-norm by interpolation.
}

\section{Introduction}
 	We study the advection-hyperdiffusion equation in $\R^n$ given by
	\begin{equation}
		\begin{aligned}
			\partial_t \theta + u\cdot \nabla \theta + \kappa \Delta^2 \theta &= 0 & \textnormal{ in }\mathbb{R}^n &\times (0,\infty),\\
			\nabla \cdot u &= 0 & \textnormal{ in }\mathbb{R}^n &\times (0,\infty),\\
			\theta(x,0) &= \theta_0(x) & \textnormal{ in }\mathbb{R}^n &,
		\end{aligned}
		\label{advectionHyperdiffusionEquation}
	\end{equation}
	where $\kappa>0$ is the molecular diffusion coefficient and $n=2,3$. This equation 
        models a hyperdiffusive scalar concentration field $\theta$ advected by a time dependent incompressible vector field $u$. Analytical and numerical mixing estimates are derived by Miles and Doering in  \cite{milesDoering2018DiffusionLimitedMixingByIncompressibleFlows} for the classical advection-diffusion equation, obtained by replacing $\Delta^2$  with $-\Delta$ in \eqref{advectionHyperdiffusionEquation}. Their numerical simulations indicate that the "filamentation length"
	\begin{equation*}
	    \lambda = \frac{\|\nabla^{-1}\theta\|_{2}}{\|\theta\|_2},
	\end{equation*}
	where
	\begin{equation*}
	    \|\nabla^{-1}\theta\|_{2}^2 = \int_{\mathbb{R}^n}\left(|\xi|^{-1}|\hat\theta|\right)^2 d\xi
	\end{equation*}
         reaches asymptotically (in time) a minimal scale, called the Batchelor scale. 
	The filamentation length $\lambda$ defined by Miles and Doering is a measure of mixing.
        Different mixing measures have been introduced in the last few years (see the excellent review paper \cite{thiffeaultUsingMultiscaleNormsToQuantifyMixingAndTransport} for an overview on this topic), but the $H^{-1}$-norm proposed in \cite{mathewMezicPetzoldAmultiscaleMeasureForMixing}, is the one that better captures the mixing mechanism as it emphasises on the role of large wave lengths in comparison to small wave lengths.

       The rigorous proof for the convergence of $\lambda$ to the Batchelor scale on bounded domains is a hard problem and only few partial results are available. In particular we want to cite the work in \cite{seisOnTheLittlewoodPaleySpectrumForPassiveScalarTransportEquations}, where the author derives $L^1$ estimates on Littlewood–Paley decompositions for linear advection–diffusion equations on the torus. Using the Fourier-splitting method introduced by Schonbeck \cite{schonbek1980DecayOfSolutionToParabolicConservationLaws}, Pottel and the second author in \cite{nobiliPottel2021lowerBoundsOnMixingNormsForTheAdvectinDiffusionEquationInRd} showed that in the whole space and under suitable conditions on the initial data, depending on decay rate assumptions on  $\|u(t)\|_2$ and  $\|\nabla u(t)\|_\infty$,  the filamentation length $\lambda$ either diverges to infinity or is bounded from below by a function that converges to zero. In this second scenario, mixing is possible. 

    The problem of mixing for the advection-hyperdiffusion equation has already been addressed in a few works. Here we want to mention the work in  \cite{fengFengIyerThiffeaultPhaseSeparationInTheAdvectiveCahnHilliardEquation}, where the authors considered the advective Cahn-Hilliard equation
	\begin{equation}
	    \label{advectiveCahnHilliard}
	    \partial_t \theta + u\cdot \nabla \theta + \kappa \Delta^2 \theta = \Delta (\theta^3-\theta),
	\end{equation}
 on the $n$-dimensional torus, and study the effect of stirring on spontaneous phase separation. For $u=0$ and small $\kappa$ solutions to \eqref{advectiveCahnHilliard} can spontaneously separate into regions of high and low concentration. On the other hand, if $u$ is sufficiently mixing, the authors show that the separation effect will be dominated and the solution converges exponentially to a homogeneous mixed state. In \cite{feng2022enhanced} the authors considered  \eqref{advectionHyperdiffusionEquation} where $u$ is an incompressible flow with circular or cylindrical symmetry in 2 and 3 space dimensions, respectively, and studied \textit{enhanced dissipation}, a phenomena intrinsically connected to mixing (see \cite{zelati2020relation} where the authors establish a precise connection between quantitative mixing rates in terms of decay of negative Sobolev norms
and enhanced dissipation time-scales). Exploiting the enhanced dissipation arising
from the combined action of the hyper-diffusion and the advection to control both the nonlinearity as well as the destabilizing effect of the negative Laplacean at large scale, the authors in \cite{mazzucato2021global} prove global existence for the modified two-dimensional Kuramoto-Sivashinsky equations in the $2-$dimensional torus 
$$\partial_t \phi+u(y)\partial_x\phi+\frac{\nu}{2}|\nabla \phi|^2+\nu\Delta^2\phi+\nu\Delta\phi=0\,.$$
 
    In this paper, we are interested in understanding how the lower bounds for $\lambda$ change when the diffusion operator in the advection-diffusion equation is substituted by the bilaplacian. We apply the methods of \cite{nobiliPottel2021lowerBoundsOnMixingNormsForTheAdvectinDiffusionEquationInRd} to bound solutions of the advection-hyperdiffusion equation.
	Our main result is the lower bound for the energy of the solution given suitable initial values and decay estimates on the flow.
    \begin{theorem}[Lower Bound for the Energy]
        \label{theoremLowerBoundForTheEnergy}
        Let the initial data satisfy
        \begin{equation*}
            \theta_0 \in L^1\cap L^2
        \end{equation*}
        and
        \begin{equation*}
            |\hat \theta_0(\xi)|\geq M
        \end{equation*}
        for $|\xi|\leq \delta$ and some constant $M>0$. Moreover assume
        \begin{equation}
            \label{u2_decay_assumption}
            \|u(t)\|_2 \sim (1+t)^{-\alpha}
        \end{equation}
        with $\alpha > \frac{3}{4}$. Then for sufficiently large $t$ the solution $\theta$ of \eqref{advectionHyperdiffusionEquation} is bounded from below by
        \begin{equation*}
            \|\theta(t)\|_2 \geq\frac{\omega_{n-1}^\frac{1}{2}}{2n^{\frac{1}{2}}} M \delta^{\frac{n}{2}}e^{-\kappa \delta^4} (1+t)^{-\frac{n}{8}},
        \end{equation*}
        where $\omega_{n-1}$ is the measure of the $(n-1)$-sphere.
    \end{theorem}
    
    As a consequence of the energy bound one gets a lower bound for the filamentation length and the $\mathring H^{-1}$-norm.
    
    \begin{theorem}[Mixing Bound]
        \label{theoremMixingBound}
        Let the assumptions of Theorem \ref{theoremLowerBoundForTheEnergy} be satisfied and additionally suppose $\theta_0\in H^1$ and
        \begin{equation}
            \label{gradUinfty_decay_assumption}
            \|\nabla u(t)\|_\infty = c_{\text{\tiny{$\nabla u$}}} (1+t)^{-\nu}.
        \end{equation}
        Then there exist constants $c_1,c_2>0$ depending on $c_{\text{\tiny{$\nabla u$}}}$, $n$, $M$, $\alpha$, $\delta$, $\kappa$, $\|\theta_0\|_1$, $\|\theta_0\|_2$ and $\|\nabla\theta_0\|_2$ such that for sufficiently large $t$
    	\begin{equation}
    		\begin{aligned}
            \|\nabla^{-1}\theta(t)\|_2 &\geq c_1 (1+t)^{-\frac{n}{8}+\frac{1}{4}} f_{\text{\tiny{$\nabla u$}}} (t),\\
            \lambda &\geq c_2 (1+t)^{\frac{1}{4}} f_{\text{\tiny{$\nabla u$}}}(t),
            \end{aligned}
        \end{equation}
        where
        \begin{equation}
            \label{lkjfdlgkjbdflk}
            f_{\text{\tiny{$\nabla u$}}}(t) = 
                \begin{cases}
                    e^{-\frac{c_{\text{\tiny{$\nabla u$}}}}{1-\nu}\left((1+t)^{1-\nu}-1\right)} & \textnormal{ for }0<\nu<1,\\
                    1 & \textnormal{ for }\nu=1,\\
                    (1+t)^{c_{\text{\tiny{$\nabla u$}}}}e^{-\frac{c_{\text{\tiny{$\nabla u$}}}}{\nu-1}\left(1-(1+t)^{-(\nu-1)}\right)} & \textnormal{ for }\nu>1.
                \end{cases}
        \end{equation}
        \vspace{40pt}   
    \end{theorem}
    The long time behaviour of these bounds is given in Table \ref{fig:longTimeBehaviourOfMixingBounds}, where
	\begin{equation*}
		\begin{aligned}
        \|\nabla^{-1}\theta\|_2 &\geq f(t) \overset{t\rightarrow\infty}{\longrightarrow} f_\infty,\\
        \lambda &\geq g(t) \overset{t\rightarrow\infty}{\longrightarrow} g_\infty.
        \end{aligned}
    \end{equation*}
    \begin{table}[ht]
        \begin{center}
            \begin{tabular}{ r | c l c l c l }
                 & & $0<\nu<1$ & & $\phantom{f_\infty}\nu=1$ & & $\phantom{f_\infty} \nu>1$\\
                 \hline
                 $n=2$ & & \begin{tabular}{@{}l@{}}$f_\infty=0$\\$g_\infty=0$\end{tabular}  & & \begin{tabular}{@{}l@{}}$\phantom{v}f_\infty=\textnormal{const}$\\$\phantom{v}g_\infty=\infty$\end{tabular} & & \begin{tabular}{@{}l@{}}$\phantom{v}f_\infty=\infty$\\$\phantom{v}g_\infty=\infty$\end{tabular} \\
                 \hline
                 $n=3$ & & \begin{tabular}{@{}l@{}} $f_\infty=0\phantom{\begin{cases}a\\b\\c\end{cases}}$ \\ $g_\infty=0$\end{tabular}  & & \begin{tabular}{@{}l@{}} $\phantom{v}f_\infty=0\phantom{\begin{cases}a\\b\\c\end{cases}}$ \\ $\phantom{v}g_\infty=\infty$\end{tabular} & & \begin{tabular}{@{}l@{}}$\phantom{v}f_\infty=\begin{cases} 0   &\textnormal{ if }c_{\text{\tiny{$\nabla u$}}}<\frac{1}{8}\\\textnormal{const} &\textnormal{ if }c_{\text{\tiny{$\nabla u$}}}=\frac{1}{8}\\ \infty  &\textnormal{ if }c_{\text{\tiny{$\nabla u$}}}>\frac{1}{8}\end{cases}$\\$\phantom{v}g_\infty=\infty$\end{tabular} 
            \end{tabular}
            \caption{Long time behaviour of the mixing bounds.}
            \label{fig:longTimeBehaviourOfMixingBounds}
        \end{center}
    \end{table}

The paper is organized as follow: 
In Section \ref{sectionUpperBoundForTheAdvectionHyperdiffusionEquation} we introduce the Fourier-splitting method, which consists in splitting the whole Fourier space in a ball with decreasing in time radius and its complement. This allows us to estimate the time derivative of the $L^2$-norm of $\theta$ on the whole space by the $L^2$-norm in this ball. Integration in time then yields an upper bound on $\|\theta(t)\|_2$, under some specific decay assumptions on $\|u(t)\|_2$.

In Section \ref{LowerBoundForTheAdvectionHyperdiffusionEquation} we prove lower bounds for $\|\theta(t)\|_2$ by viewing the solution of \eqref{advectionHyperdiffusionEquation} as a perturbation of the hyperdiffusion equation.
In particular, in Subsection \ref{subsectionLowerBoundForTheEnergy} we show that, under suitable decay assumptions on the stirring field, the perturbation decays at least as fast as the solution of the hyperdiffusion equation, yielding the result Theorem \ref{theoremLowerBoundForTheEnergy}.

Finally, in Section \ref{sectionLowerBoundsOnTheMixingNorms} we derive upper bounds for the gradient of the solution yielding the result in Theorem \ref{theoremMixingBound} by interpolation.

\textbf{Notation:} Norms are considered over the whole space unless differently stated. $a\lesssim b$ denotes $a \leq c b$ for a constant $c>0$. Similarly $a\sim b$ and $a\gtrsim b$ with $a=cb$ and $a\geq cb$. As we did not attempt to optimize the constant prefactors in the upper/lower bounds, the dependency of constants on parameters like the dimension $n$, the molecular diffusivity $\kappa$ or the decay exponents $\alpha$ and $\nu$ (see \eqref{u2_decay_assumption},\eqref{gradUinfty_decay_assumption}) will often be hidden to increase the readability of the manuscript, unless tracking is required for the argument to work. Note that in few estimates these constants that we do not want to track may change from line to line, without being renamed.

\section{Upper Bound for the Advection-Hyperdiffusion Equation}
\label{sectionUpperBoundForTheAdvectionHyperdiffusionEquation}
    In this section we introduce the Fourier-splitting method and prove upper bounds for the advection-hyperdiffusion equation under suitable decay of the flow field. This bound will later be used to derive upper bounds for the perturbation and the gradient of the solution.
    \begin{lemma}[Upper Bound for the Advection-Hyperdiffusion Equation]
    \label{lemmaUpperBoundAdvectionHyperdiffusion}
    Let the initial data satisfy 
    \begin{equation*}
        \theta_0\in L^1 \cap L^2
    \end{equation*}
    and let $u(\cdot, t)\in L^2$ be a divergence-free vector field satisfying
    \begin{equation}
        \label{lemmaUpperBoundAdvectionHyperdiffusionuL2DecayAssumption}
        \|u(t)\|_2 \sim (1+t)^{-\alpha}
    \end{equation}
    with $\alpha>\frac{6-n}{8}$. Then there exists a constant $c>0$ depending on $n$, $\alpha$, $\kappa$, $\|\theta_0\|_1$ and $\|\theta_0\|_2$ such that
    \begin{equation*}
        \|\theta(t)\|_2 \leq c (1+t)^{-\frac{n}{8}}.
    \end{equation*}
    \end{lemma}
    
    \begin{proof}
    The proof consists of $3$ steps. First we bound the derivative of the $L^2$-norm of $\theta$ over the whole space by the $L^2$-norm of $\theta$ over a ball with radius decreasing in time. Second we show, by integrating this ODE in time, that if $\theta$ decays with some rate in a specific range then it actually decays faster. Third we iterate to get the maximal decay rate of this range resulting in the final bound.
    \paragraph{Step 1:}
    Testing \eqref{advectionHyperdiffusionEquation} with $\theta$, using partial integration, one gets
    \begin{equation}
        \label{lemmaUpperBoundAdvectionHyperdiffusionX}
        \frac{1}{2}\frac{d}{dt} \|\theta\|_2^2 = - \kappa \int_{\mathbb{R}^n}\theta \Delta^2 \theta \ dx - \int_{\mathbb{R}^n}\theta u \cdot \nabla \theta \ dx = - \kappa \|\Delta \theta\|_2^2,
    \end{equation}
    where the advection term vanishes by the incompressibility condition of $u$ as
    \begin{equation*}
        \int_{\mathbb{R}^n}\theta u\cdot \nabla \theta \ dx = -\int_{\mathbb{R}^n}\theta^2\nabla \cdot u \ dx - \int_{\mathbb{R}^n}\theta u\cdot\nabla \theta \ dx = - \int_{\mathbb{R}^n}\theta u\cdot\nabla \theta \ dx.
    \end{equation*}
    By Plancherel \eqref{lemmaUpperBoundAdvectionHyperdiffusionX} can be expressed in Fourier space via
    \begin{equation}
        \label{lemmaUpperBoundAdvectionHyperdiffusionA}
        \frac{d}{dt} \|\hat \theta\|_2^2 = - 2\kappa \|\xi^2 \hat \theta\|_2^2.
    \end{equation}
    Define the set
    \begin{equation*}
        S(t) = \left\lbrace \xi\in \mathbb{R}^n \ \middle |\  |\xi| \leq r(t) \right \rbrace
    \end{equation*}
    with $r=\left(\frac{\beta}{2\kappa(1+t)}\right)^\frac{1}{4}$, where $\beta>0$ is a constant to be specified later and denote its complement by $S^c$. Then by \eqref{lemmaUpperBoundAdvectionHyperdiffusionA}
    \begin{equation*}
        \begin{aligned}
            \frac{d}{dt} \|\hat \theta\|_2^2
            &= - 2\kappa \|\xi^2 \hat \theta\|_2^2
            \leq - 2\kappa \|\xi^2 \hat \theta\|_{L^2(S^c)}^2
            = -2 \kappa \int_{S^c} |\xi|^4 |\hat \theta|^2 d\xi
            \\
            &\leq - 2\kappa r^4 \int_{S^c} |\theta|^2 d\xi
            = -2\kappa r^4 \left(\|\hat\theta\|_2^2 - \|\hat\theta \|_{L^2(S)}^2 \right)
            \\
            &= - \frac{\beta}{1+t} \left(\|\hat\theta\|_2^2 - \|\hat \theta \|_{L^2(S)}^2\right)
        \end{aligned}
    \end{equation*}
    such that
    \begin{equation}
        \label{lemmaUpperBoundAdvectionHyperdiffusionE}
        \begin{aligned}
            \frac{d}{dt} \left((1+t)^\beta \|\hat\theta\|_2^2\right) &= \beta (1+t)^{\beta-1} \|\hat\theta\|_2^2 + (1+t)^\beta \frac{d}{dt} \|\hat \theta\|_2^2 
            \\
            &\leq \beta (1+t)^{\beta-1} \|\hat\theta\|_2^2 - \beta (1+t)^{\beta-1}\left(\|\hat\theta\|_2^2 - \|\hat \theta \|_{L^2(S)}^2\right) 
            \\
            &= \beta (1+t)^{\beta-1} \|\hat\theta\|_{L^2(S)}^2.
        \end{aligned}
    \end{equation}
    \paragraph{Step 2:}
    Assume
    \begin{equation}
        \label{lemmaUpperBoundAdvectionHyperdiffusionF}
        \|\theta\|_2 \lesssim A (1+t)^{-\gamma}
    \end{equation}
    with $A$ depending on $\kappa$, $n$, $\|\theta_0\|_1$ and $\|\theta_0\|_2$ for some $\gamma\geq 0$. Notice that by \eqref{lemmaUpperBoundAdvectionHyperdiffusionA} the assumption is fulfilled for $\gamma=0$. Writing \eqref{advectionHyperdiffusionEquation} in Fourier space
    \begin{equation}
        \label{lemmaUpperBoundAdvectionHyperdiffusionB}
        \hat\theta_t = -\kappa |\xi|^4 \hat \theta - \widehat{u\cdot\nabla \theta}
    \end{equation}
    and its solution is given by
    \begin{equation*}
        \hat\theta(t,\xi) = e^{-\kappa |\xi|^4 t} \hat\theta_0(\xi) - \int_0^t e^{-\kappa |\xi|^4 (t-s)} \widehat{u\cdot \nabla \theta}(s,\xi) \ ds.
    \end{equation*}
    Young's inequality yields
    \begin{equation}
        \label{lemmaUpperBoundAdvectionHyperdiffusionC}
        \begin{aligned}
            |\hat\theta(\xi)|^2 &= e^{-2\kappa |\xi|^4 t} |\hat\theta_0(\xi)|^2 - 2e^{-\kappa |\xi|^4 t} \hat\theta_0(\xi)\int_0^t e^{-\kappa |\xi|^4 (t-s)} \widehat{u\cdot \nabla \theta}(s,\xi) \ ds
            \\
            &\qquad +\left(\int_0^t e^{-\kappa |\xi|^4 (t-s)} \widehat{u\cdot \nabla \theta}(s,\xi) \ ds\right)^2
            \\
            &\leq 2e^{-2\kappa |\xi|^4 t} |\hat\theta_0(\xi)|^2+2\left(\int_0^t e^{-\kappa |\xi|^4 (t-s)} \widehat{u\cdot \nabla \theta}(s,\xi) \ ds\right)^2
            \\
            &\leq 2e^{-2\kappa |\xi|^4 t} |\hat\theta_0(\xi)|^2+2\left(\int_0^t \left|\widehat{u\cdot \nabla \theta}\right|(s,\xi) \ ds\right)^2.
        \end{aligned}
    \end{equation}
    Using the incompressibility of $u$, the advection term can be estimated by
    \begin{equation}
        \label{lemmaUpperBoundAdvectionHyperdiffusionG}
        \begin{aligned}
            \left|\widehat{u\cdot \nabla \theta}\right|(\xi)
            &= \left| \int_{\mathbb{R}^n} u \cdot \nabla \theta e^{i\xi \cdot x} dx \right|
            = \left| \int_{\mathbb{R}^n}u \theta \cdot \nabla e^{i\xi \cdot x} dx\right|
            = \left| \int_{\mathbb{R}^n}u \theta \cdot i\xi  e^{i\xi \cdot x} dx\right|
            \\
            &
            \leq |\xi|\ \|u\|_2 \|\theta\|_2
            \lesssim  A |\xi| (1+t)^{-\alpha-\gamma},
        \end{aligned}
    \end{equation}
    where in the last inequality we used assumptions \eqref{lemmaUpperBoundAdvectionHyperdiffusionuL2DecayAssumption} and \eqref{lemmaUpperBoundAdvectionHyperdiffusionF}.
    Using $\ln(1+t)\lesssim \frac{1}{\epsilon}(1+t)^\epsilon$ with $0<\epsilon \ll 1$, the integral on the right-hand side of \eqref{lemmaUpperBoundAdvectionHyperdiffusionC} can be bounded by
    \begin{equation}
        \label{lemmaUpperBoundAdvectionHyperdiffusionD}
        \begin{aligned}
            \int_0^t \left|\widehat{u\cdot \nabla \theta}\right|(s,\xi) \ ds &\leq A|\xi| \int_0^t (1+s)^{-\alpha-\gamma}\ ds 
            \\
            &= A|\xi| 
                \begin{cases}
                    \frac{1}{1-(\alpha+\gamma)} \left((1+t)^{1-(\alpha+\gamma)}-1\right) & \textnormal{ for } \alpha+\gamma < 1\\
                    \ln(1+t) & \textnormal{ for } \alpha+\gamma = 1\\
                    \frac{1}{\alpha+\gamma-1} \left(1-(1+t)^{-(\gamma+\alpha)+1}\right) & \textnormal{ for } \alpha+\gamma > 1
                \end{cases}
            \\
            &\leq A|\xi| 
                \begin{cases}
                    \frac{1}{1-(\alpha+\gamma)} (1+t)^{1-(\alpha+\gamma)} & \textnormal{ for } \alpha+\gamma < 1\\
                    \frac{1}{\epsilon}(1+t)^{\epsilon} & \textnormal{ for } \alpha+\gamma = 1\\
                    \frac{1}{\alpha+\gamma-1}  & \textnormal{ for } \alpha+\gamma > 1
                \end{cases}
            \\
            &\leq A |\xi| h_{\alpha+\gamma} (t),
        \end{aligned}
    \end{equation}
    where
    \begin{equation*}
        h_{\alpha+\gamma}(t) = 
            \begin{cases}
                    \frac{1}{1-(\alpha+\gamma)} (1+t)^{1-(\alpha+\gamma)} & \textnormal{ for } \alpha+\gamma < 1\\
                    \frac{1}{\epsilon}(1+t)^{\epsilon} & \textnormal{ for } \alpha+\gamma = 1\\
                    \frac{1}{\alpha+\gamma-1} & \textnormal{ for } \alpha+\gamma > 1
            \end{cases}
    \end{equation*}
    Calculating
    \begin{equation}
        \label{lemmaUpperBoundAdvectionHyperdiffusionH}
        \begin{aligned}
            \|\xi\|_{L^2(S)}^2 &= \int_{|\xi|\leq r(t)} |\xi|^2 \ d\xi = \omega_{n-1} \int_0^{r(t)} \rho^{n+1} d\rho = \frac{\omega_{n-1}}{n+2} r^{n+2} 
            \\
            &= \frac{\omega_{n-1}}{n+2} \left(\frac{\beta}{2\kappa(1+t)}\right)^{\frac{n+2}{4}}
        \end{aligned}
    \end{equation}
    we can integrate \eqref{lemmaUpperBoundAdvectionHyperdiffusionC} over $S$ and, using \eqref{lemmaUpperBoundAdvectionHyperdiffusionD} and Lemma \ref{appendixLemmaBoundsOnHyperdiffusionEquation} (proven in the Appendix), find
    \begin{equation}
        \label{lemmaUpperBoundAdvectionHyperdiffusionI}
        \begin{aligned}
            \|\hat\theta\|_{L^2(S)}^2 &= 2 \left\|e^{-\kappa|\xi|^4t}\hat\theta_0\right\|_{L^2(S)}^2 + 2  \left\| \int_0^t \left|\widehat{u\cdot \nabla \theta}\right|(s,\xi) \ ds \right\|_{L^2(S)}^2 
            \\
            &\lesssim \left\|e^{-\kappa|\xi|^4t}\hat\theta_0\right\|_{L^2(S)}^2 + A^2 \left\| \xi \right\|_{L^2(S)}^2  h_{\alpha+\gamma}^2(t)
            \\
            &\lesssim \left\|e^{-\kappa|\xi|^4t}\hat\theta_0\right\|_{2}^2 + A^2 \left(\frac{\beta}{\kappa}\right)^{\frac{n+2}{4}} (1+t)^{-\frac{n+2}{4}}  h_{\alpha+\gamma}^2(t)
            \\
            &\lesssim (\kappa t)^{-\frac{n}{4}}\|\theta_0\|_1^2 + A^2 \left(\frac{\beta}{\kappa}\right)^{\frac{n+2}{4}} (1+t)^{-\frac{n+2}{4}}  h_{\alpha+\gamma}^2(t).
        \end{aligned}
    \end{equation}
    Plugging \eqref{lemmaUpperBoundAdvectionHyperdiffusionI} into \eqref{lemmaUpperBoundAdvectionHyperdiffusionE} one gets
    \begin{equation*}
        \begin{aligned}
            \frac{d}{dt} \left((1+t)^\beta \|\hat\theta\|_2^2\right) &= \beta (1+t)^{\beta-1} \|\hat\theta\|_{L^2(S)}^2 
            \\
            &\lesssim \beta \kappa^{-\frac{n}{4}} (1+t)^{\beta-1-\frac{n}{4}} \|\theta_0\|_1^2 
            \\
            &\qquad+ A^2 \beta^{\frac{n+6}{4}} \kappa^{-\frac{n+2}{4}} (1+t)^{\beta-\frac{n+6}{4}} h_{\alpha+\gamma}^2(t).
        \end{aligned}
    \end{equation*}
    Choosing $\beta$ big enough such that the right-hand side becomes a polynomial with positive exponent and integrating in time yields
    \begin{equation*}
        \begin{aligned}
            &(1+t)^\beta\|\hat\theta\|_2^2 - \|\hat\theta_0\|_2^2 
            \\
            &\qquad\lesssim \kappa^{-\frac{n}{4}} \frac{\beta}{\beta-\frac{n}{4}} \|\hat\theta_0\|_1^2 \left((1+t)^{\beta-\frac{n}{4}}-1\right) 
            \\
            &\qquad\qquad+ A^2 \kappa^{-\frac{n+2}{4}}\beta^{\frac{n+6}{4}}
            \begin{cases}
                    \frac{1}{1-(\alpha+\gamma)}\frac{1}{\beta-\frac{n}{4}-2(\alpha+\gamma)+\frac{3}{2}} \left((1+t)^{\beta-\frac{n}{4}-2(\alpha+\gamma)+\frac{3}{2}}-1\right)\\
                    \epsilon^{-2} \frac{1}{2\epsilon+\beta-\frac{n}{4}-\frac{1}{2}}\left((1+t)^{2\epsilon+\beta-\frac{n}{4}-\frac{1}{2}}-1\right)\\
                    \frac{1}{\alpha+\gamma-1} \frac{1}{\beta-\frac{n}{4}-\frac{1}{2}}\left((1+t)^{\beta-\frac{n}{4}-\frac{1}{2}}-1\right)
            \end{cases}
        \end{aligned}
    \end{equation*}
    such that dividing by $(1+t)^\beta$ results in
    \begin{equation*}
        \begin{aligned}
            &\|\hat\theta\|_2^2 \lesssim \|\hat\theta_0\|_2^2 (1+t)^{-\beta} + \kappa^{-\frac{n}{4}} \frac{\beta}{\beta-\frac{n}{4}} \|\hat\theta_0\|_1^2 (1+t)^{-\frac{n}{4}}
            \\
            &\qquad\qquad+ A^2 \kappa^{-\frac{n+2}{4}}\beta^{\frac{n+6}{4}}
            \begin{cases}
                    \frac{1}{1-(\alpha+\gamma)}\frac{1}{\beta-\frac{n}{4}-2(\alpha+\gamma)+\frac{3}{2}} (1+t)^{-\frac{n}{4}-2(\alpha+\gamma)+\frac{3}{2}}\\
                    \epsilon^{-2} \frac{1}{2\epsilon+\beta-\frac{n}{4}-\frac{1}{2}}(1+t)^{2\epsilon-\frac{n}{4}-\frac{1}{2}}\\
                    \frac{1}{\alpha+\gamma-1} \frac{1}{\beta-\frac{n}{4}-\frac{1}{2}}(1+t)^{-\frac{n}{4}-\frac{1}{2}}
            \end{cases}
        \end{aligned}
    \end{equation*}
    and for $\alpha+\gamma<\frac{3}{4}$ 
    \begin{equation*}
        \|\hat\theta\|_2^2\leq c (1+t)^{-\frac{n}{4}-2(\alpha+\gamma)+\frac{3}{2}},
    \end{equation*}
    while for $\alpha+\gamma\geq \frac{3}{4}$
    \begin{equation*}
        \|\hat\theta\|_2^2\leq c (1+t)^{-\frac{n}{4}},
    \end{equation*}
    where in both cases the constant depends on $n$, $\alpha$, $\gamma$, $\kappa$, $\|\theta\|_1$ and $\|\theta\|_2$.
    \paragraph{Step 3:}
    In the previous Step we have shown that assuming $\|\theta\|_2$ decays with rate $\gamma_0$ it actually decays with rate $\gamma_1=\frac{n}{8}+\alpha+\gamma_0-\frac{3}{4}$ if $\alpha+\gamma_0<\frac{3}{4}$ or with rate $\gamma_1=\frac{n}{8}$ if $\alpha+\gamma_0\geq\frac{3}{4}$. Iterating through this process one finds
    \begin{equation*}
        \gamma_{j+1} = \frac{n}{8} - \frac{3}{4} + \alpha+\gamma_j
    \end{equation*}
    such that the condition $\alpha>\frac{6-n}{8}$ ensures $\gamma_j$ is a strictly increasing sequence resulting in the desired decay rate.
    \end{proof}

\section{Lower Bound for the Advection-Hyperdiffusion Equation}
\label{LowerBoundForTheAdvectionHyperdiffusionEquation}
    For the lower bound on the solution of the advection-hyperdiffusion equation we first prove a lower bound for the hyperdiffusion equation and then show that the perturbation that arises because of the advection decays at least with the same rate.
    
\subsection{Lower Bound for the Hyperdiffusion Equation}
    \begin{lemma}[Lower Bound for the Hyperdiffusion Equation]
        \label{lemmaLowerBoundHyperdiffusionEquation}
        Let $T$ solve
        \begin{equation}
            \label{lemmaLowerBoundHyperdiffusionEquationPDEforT}
            \begin{aligned}
                \partial_t T +\kappa \Delta^2 T&=0 &\textnormal{ in }&\mathbb{R}^n \times (0,\infty),\\
                T(0,x)&=\theta_0(x) &\textnormal{ in }&\mathbb{R}^n
            \end{aligned}
        \end{equation}
        and $\theta_0\in L^2$ satisfy
        \begin{equation}
            \label{lemma1AssumptionThetaZero}
            |\hat \theta_0(\xi)|\geq M
        \end{equation}
        for $|\xi|\leq \delta$. Then for $t\geq 1$
        \begin{equation*}
            \|T(t)\|_2 \geq \omega_{n-1}^\frac{1}{2} n^{-\frac{1}{2}} M\delta^{\frac{n}{2}} e^{-\kappa\delta^4} (1+t)^{-\frac{n}{8}},
        \end{equation*}
        where $\omega_{n-1}$ is the measure of the $n-1$-sphere.
    \end{lemma}

    \begin{proof}
        As the Fourier transform of \eqref{lemmaLowerBoundHyperdiffusionEquationPDEforT} is given by
        \begin{equation*}
            \begin{aligned}
                \partial_t \hat T + \kappa |\xi|^4 \hat T &= 0 &\textnormal{ in }&\mathbb{R}^n \times (0,\infty),\\
                \hat T(0,\xi)&=\hat \theta_0(\xi) &\textnormal{ in }&\mathbb{R}^n
            \end{aligned}
        \end{equation*}
        its solution can be represented by
        \begin{equation*}
            \hat T (t,\xi) = \hat \theta_0(\xi) e^{-\kappa |\xi|^4 t},
        \end{equation*}
        which by Plancherel and assumption \eqref{lemma1AssumptionThetaZero} yields
        \begin{equation}
            \label{lemmaLowerBoundHyperdiffusionEquationPDEforTa}
            \begin{aligned}
                \int_{\mathbb{R}^n}|T|^2 dx
                &= \int_{\mathbb{R}^n}|\hat T|^2 d\xi
                = \int_{\mathbb{R}^n}|\hat \theta_0|^2 e^{-2\kappa |\xi|^4 t} d\xi
                \geq M^2 \int_{|\xi|\leq \delta} e^{-2\kappa |\xi|^4 t}\ d\xi.
            \end{aligned}
        \end{equation}
        By passing to spherical coordinates and the change of variables $\mu = (2\kappa t)^\frac{1}{4}r$ the integral can be expressed by
        \begin{equation}
            \label{lemmaLowerBoundHyperdiffusionEquationPDEforTb}
            \begin{aligned}
                \int_{|\xi|\leq \delta} e^{-2\kappa |\xi|^4 t}\ d\xi &= \omega_{n-1} \int_0^\delta e^{-2\kappa r^4 t} r^{n-1} \ dr 
                \\
                &=  \omega_{n-1} (2\kappa t)^{-\frac{n}{4}} \int_0^{(2\kappa t)^\frac{1}{4}\delta} e^{-\mu^4} \mu^{n-1}\ d\mu,    
            \end{aligned}
        \end{equation}
        where $\omega_{n-1}$ is the measure of the $n-1$-sphere. For $t\geq 1$ we estimate this integral by
        \begin{equation}
            \label{lemmaLowerBoundHyperdiffusionEquationPDEforTc}
            \begin{aligned}
                \int_0^{(2\kappa t)^\frac{1}{4}\delta} e^{-\mu^4} \mu^{n-1}\ d\mu &\geq \int_0^{(2\kappa)^\frac{1}{4}\delta} e^{-\mu^4} \mu^{n-1}\ d\mu \geq  e^{-2 \kappa \delta^4} \int_0^{(2\kappa)^\frac{1}{4}\delta}  \mu^{n-1}\ d\mu 
                \\
                &= \frac{1}{n} e^{-2 \kappa \delta^4} (2\kappa)^\frac{n}{4} \delta^n.
            \end{aligned}
        \end{equation}
        Combining \eqref{lemmaLowerBoundHyperdiffusionEquationPDEforTa}, \eqref{lemmaLowerBoundHyperdiffusionEquationPDEforTb} and \eqref{lemmaLowerBoundHyperdiffusionEquationPDEforTc} yields the claim as
        \begin{equation*}
            \begin{aligned}
                \int_{\mathbb{R}^n}|T|^2 dx
                &\geq M^2 \int_{|\xi|\leq \delta} e^{-2\kappa |\xi|^4 t}\ d\xi
                = M^2 \omega_{n-1} (2\kappa t)^{-\frac{n}{4}} \int_0^{(2\kappa t)^\frac{1}{4}\delta} e^{-\mu^4} \mu^{n-1}\ d\mu
                \\
                &\geq M^2 \frac{\omega_{n-1}}{n} e^{-2\kappa \delta^4} \delta^n (1+t)^{-\frac{n}{4}}.
            \end{aligned}
        \end{equation*}
    \end{proof}
    
\subsection{Upper Bound for the Perturbation}
    \begin{lemma}
    \label{lemmaUpperBoundForThePerturbation}
    The perturbation $\eta = \theta - T$ solves
    \begin{equation}
        \label{lemmaUpperBoundForThePerturbationPDEforEta}
        \begin{aligned}
            \partial_t \eta + \kappa \Delta^2 \eta + u\cdot \nabla \theta &= 0 & \textnormal{ in }\mathbb{R}^n &\times (0,\infty),\\
    		\eta(x,0) &= 0 & \textnormal{ in }\mathbb{R}^n &.
        \end{aligned}
    \end{equation}
    Suppose the assumptions of Theorem \ref{theoremLowerBoundForTheEnergy} are fulfilled. Then there exists a constant $c>0$ depending on $n,$ $\alpha$, $\kappa$, $\|\theta_0\|_1$ and $\|\theta_0\|_2$ such that
    \begin{equation*}
        \begin{aligned}
            \|\eta\|_2 &\leq c (1+t)^{\frac{3}{8}-\frac{n}{8}-\frac{\alpha}{2}} & &\textnormal{ if } 1-\frac{n}{8} < \alpha < \frac{5}{4}\\
            \|\eta\|_2 &\leq c (1+t)^{-\frac{1}{4}-\frac{n}{8}} & &\textnormal{ if } \alpha \geq \frac{5}{4}.
        \end{aligned}
    \end{equation*}
    \end{lemma}
    The proof works similar to Lemma \ref{lemmaUpperBoundAdvectionHyperdiffusion} based on the Fourier-splitting method.
    \begin{proof}
    Testing \eqref{lemmaUpperBoundForThePerturbationPDEforEta} with $\eta$, using partial integration and $\eta = \theta - T$, yields
    \begin{equation*}
        \begin{aligned}
            \frac{1}{2}\frac{d}{dt} \|\eta\|_2^2
            &= - \kappa \int_{\mathbb{R}^n}\eta \Delta^2 \eta \ dx - \int_{\mathbb{R}^n}\eta u\cdot \nabla \theta \ dx
            \\
            &= - \kappa \|\Delta\eta\|_2^2 - \int_{\mathbb{R}^n}\theta u\cdot \nabla \theta \ dx + \int_{\mathbb{R}^n}T u\cdot \nabla \theta \ dx.
        \end{aligned}
    \end{equation*}
    The second term on the right-hand side vanishes as by partial integration and the incompressibility condition on $u$ as
    \begin{equation*}
        \int_{\mathbb{R}^n}\theta u\cdot \nabla \theta \ dx
        = - \int_{\mathbb{R}^n}\theta^2 \nabla \cdot u \ dx - \int_{\mathbb{R}^n}\theta u \cdot \nabla \theta \ dx
        = - \int_{\mathbb{R}^n}\theta u \cdot \nabla \theta \ dx
    \end{equation*}
    and the third term is given by
    \begin{equation*}
        \int_{\mathbb{R}^n}T u\cdot \nabla \theta \ dx
        = -\int_{\mathbb{R}^n}T \theta \nabla \cdot u \ dx - \int_{\mathbb{R}^n}\theta u \cdot \nabla T \ dx
        = - \int_{\mathbb{R}^n}\theta u \cdot \nabla T\ dx
    \end{equation*}
    such that Hölder's inequality yields
    \begin{equation*}
        \frac{1}{2}\frac{d}{dt} \|\eta\|_2^2
        = - \kappa \|\Delta\eta\|_2^2 + \int_{\mathbb{R}^n}T u\cdot \nabla \theta \ dx
        \leq - \kappa \|\Delta\eta\|_2^2 + \|\nabla T\|_\infty \|u\|_2 \|\theta\|_2.
    \end{equation*}
    Analogous to before define
    \begin{equation*}
        S(t) = \left\lbrace \xi\in \mathbb{R}^n \ \middle |\  |\xi| \leq r(t) \right \rbrace,
    \end{equation*}
    where $r=\left(\frac{\beta}{2\kappa(1+t)}\right)^\frac{1}{4}$ with a constant $\beta>0$. Then from Plancherel follows
    \begin{equation*}
        \begin{aligned}
            \frac{d}{dt} \|\eta\|_2^2 &= -2 \kappa \left\||\xi|^2 \hat \eta \right\|_2^2 + 2 \|\nabla T\|_\infty \|u\|_2 \|\theta\|_2
            \\
            &\leq -2 \kappa \int_{S^c} |\xi|^4 |\hat\eta|^2 d\xi + 2 \|\nabla T\|_\infty \|u\|_2 \|\theta\|_2 
            \\
            &\leq - \beta (1+t)^{-1}\int_{S^c} |\hat\eta|^2 d\xi + 2 \|\nabla T\|_\infty \|u\|_2 \|\theta\|_2
            \\
            &= - \beta (1+t)^{-1} \left(\|\hat\eta\|_2^2 - \|\hat\eta\|_{L^2(S)}^2\right) + 2 \|\nabla T\|_\infty \|u\|_2 \|\theta\|_2
        \end{aligned}
    \end{equation*}
    such that
    \begin{equation}
        \label{lemmaUpperBoundForThePerturbationProofB}
        \begin{aligned}
            \frac{d}{dt}\left((1+t)^\beta \|\eta\|_2^2\right) &= \beta (1+t)^{\beta - 1} \|\eta\|_2^2 + (1+t)^\beta \frac{d}{dt} \|\eta\|_2^2
            \\
            &= \beta (1+t)^{\beta-1} \|\hat\eta\|_{L^2(S)}^2 + 2 (1+t)^\beta \|\nabla T\|_\infty \|u\|_2 \|\theta\|_2.
        \end{aligned}
    \end{equation}
    In order to estimate the first term, the solution of \eqref{lemmaUpperBoundForThePerturbationPDEforEta} in Fourier space is given by
    \begin{equation*}
        \hat\eta(t,\xi) = - \int_0^t e^{-\kappa |\xi|^4 (t-s)}\widehat{u\cdot\nabla \theta}(\xi)\ ds
    \end{equation*}
    such that by \eqref{lemmaUpperBoundAdvectionHyperdiffusionG}, the decay assumption \eqref{lemmaUpperBoundAdvectionHyperdiffusionuL2DecayAssumption}, Lemma \ref{lemmaUpperBoundAdvectionHyperdiffusion} and \eqref{lemmaUpperBoundAdvectionHyperdiffusionH}
    \begin{equation}
        \label{lemmaUpperBoundForThePerturbationProofA}
        \begin{aligned}
            \|\hat \eta \|_{L^2(S)}^2 &= \int_S \left(\int_0^t e^{-\kappa |\xi|^4 (t-s)} \widehat{u\cdot\nabla \theta}(\xi)\ ds\right)^2 \ d\xi 
            \\
            &\leq \int_S \left(\int_0^t e^{-\kappa |\xi|^4 (t-s)} |\xi| \|u\|_2 \|\theta\|_2 \ ds\right)^2\ d\xi
            \\
            &\lesssim \int_S |\xi|^2 \left( \int_0^t (1+s)^{-\alpha-\frac{n}{8}} \ ds \right)^2 \ d\xi
            \\
            &= \int_S |\xi|^2 \left(\frac{1}{\alpha+\frac{n}{8}-1} \left(1- (1+t)^{-\alpha-\frac{n}{8}+1}\right)\right)^2 \ d\xi
            \\
            &\leq \left(\alpha+\frac{n}{8}-1\right)^{-2} \int_S |\xi|^2 \ d\xi
            \\
            &\sim \left(\frac{\beta}{\kappa (1+t)}\right)^{\frac{n}{4}+\frac{1}{2}}
        \end{aligned}
    \end{equation}
    as $\alpha> 1-\frac{n}{8}$. Combining \eqref{lemmaUpperBoundForThePerturbationProofA} and \eqref{lemmaUpperBoundForThePerturbationProofB}, Lemma \ref{appendixLemmaBoundsOnHyperdiffusionEquation} (proven in the Appendix), the decay assumption \eqref{lemmaUpperBoundAdvectionHyperdiffusionuL2DecayAssumption} and Lemma \ref{lemmaUpperBoundAdvectionHyperdiffusion} we find
    \begin{equation*}
        \begin{aligned}
            \frac{d}{dt}\left((1+t)^\beta \|\eta\|_2^2\right) &= \beta (1+t)^{\beta-1} \|\hat\eta\|_{L^2(S)}^2 + 2 (1+t)^\beta\|\nabla T\|_\infty \|u\|_2 \|\theta\|_2
            \\
            &\lesssim \beta^{\frac{n}{4}+\frac{3}{2}} \kappa^{-\frac{n}{4}-\frac{1}{2}} (1+t)^{\beta-\frac{3}{2}-\frac{n}{4}} + (1+t)^\beta\|\nabla T\|_\infty \|u\|_2 \|\theta\|_2
            \\
            &\lesssim \beta^{\frac{n}{4}+\frac{3}{2}} (1+t)^{\beta-\frac{3}{2}-\frac{n}{4}} +  (1+t)^{\beta-\frac{1}{4}-\frac{n}{4}-\alpha}.
        \end{aligned}
    \end{equation*}
    Choosing $\beta > \frac{1}{2}+\frac{n}{4}, \frac{3}{4}+\frac{n}{4}+\alpha$ and integrating in time yields
    \begin{equation*}
        \begin{aligned}
            (1+t)^\beta\|\eta\|_2^2 &\lesssim \frac{1}{\beta-\frac{n}{4}+\frac{1}{2}}\left((1+t)^{\beta-\frac{n}{4}-\frac{1}{2}}-1\right) 
            \\
            &\qquad + \frac{1}{\beta- \frac{n}{4} -\alpha +\frac{3}{4}} \left((1+t)^{\beta- \frac{n}{4} -\alpha +\frac{3}{4}}-1\right)
            \\
            &\lesssim (1+t)^{\beta-\frac{n}{4}-\frac{1}{2}} +(1+t)^{\beta- \frac{n}{4} -\alpha +\frac{3}{4}}
        \end{aligned}
    \end{equation*}
    such that dividing by $(1+t)^\beta$ results in the desired decay as
    \begin{equation*}
        \|\eta\|_2^2 \lesssim (1+t)^{-\frac{n}{4}-\frac{1}{2}} +(1+t)^{- \frac{n}{4} -\alpha +\frac{3}{4}}.
    \end{equation*}
    \end{proof}

\subsection{Lower Bound for the Energy}
\label{subsectionLowerBoundForTheEnergy}
    We are now able to prove Theorem \ref{theoremLowerBoundForTheEnergy}, i.e. the lower bound for the energy.
    \begin{proof}[Theorem \ref{theoremLowerBoundForTheEnergy}]
    As $\eta = \theta - T$ and therefore $\|T\|_2= \|\eta-\theta\|_2 \leq \|\theta\|_2+\|\eta\|_2$ the energy of the solution of the advection-hyperdiffusion equation can be bounded by $\|\theta\|_2 \geq \|T\|_2 - \|\eta\|_2$. Lemma \ref{lemmaLowerBoundHyperdiffusionEquation} yields a lower bound for $T$, while Lemma \ref{lemmaUpperBoundForThePerturbation} implies an upper bound on $\eta$ such that
    \begin{equation*}
        \begin{aligned}
            \|\theta\|_2 &\geq \|T\|_2 - \|\eta\|_2 \geq \tilde c M \delta^\frac{n}{2} e^{-\kappa \delta^4} (1+t)^{-\frac{n}{8}} + c (1+t)^{-\frac{n}{8} +\frac{3}{8} - \min\left\lbrace \frac{\alpha}{2}, \frac{5}{8} \right\rbrace}\\
            &\geq \tilde c \left(M \delta^{\frac{n}{2}}e^{-\kappa \delta^4} - c (1+t)^{\frac{3}{8}- \min\left\lbrace \frac{\alpha}{2}, \frac{5}{8} \right\rbrace} \right) (1+t)^{-\frac{n}{8}},
        \end{aligned}
    \end{equation*}
    where $\tilde c = \omega_{n-1}^\frac{1}{2} n^{-\frac{1}{2}}$. By assumption $\alpha > \frac{3}{4}$ such that for all $t>t_1$ given by
    \begin{equation*}
        M \delta^{\frac{n}{2}}e^{-\kappa \delta^4} = \frac{c}{2}(1+t_1)^{\frac{3}{8}- \min\left\lbrace \frac{\alpha}{2}, \frac{5}{8} \right\rbrace}
    \end{equation*}
    the solution decays at most by
    \begin{equation*}
        \|\theta\|_2 \geq \frac{\omega_{n-1}^\frac{1}{2}}{2n^{\frac{1}{2}}} M \delta^{\frac{n}{2}}e^{-\kappa \delta^4} (1+t)^{-\frac{n}{8}}.
    \end{equation*}
    \end{proof}

\section{Lower Bounds on the Mixing Norms}
\label{sectionLowerBoundsOnTheMixingNorms}
    In order to get lower bounds for the mixing norms we need an upper bound for the gradient of the solution as the interpolation takes the reciprocal of this bound.

\subsection{Upper Bound for the Gradient of the Advection-Hyper-diffusion Equation}
\label{subsectionUpperBoundForTheGradientOfTheAdvectionHyperdiffusionEquation}
    \begin{lemma}[Upper Bound for the Gradient of the Advection-Hyperdiffusion Equation]
    \label{lemmaUpperBoundForTheGradientOfTheAdvectionHyperdiffusionEquation}
    Let the assumptions of Theorem \ref{theoremLowerBoundForTheEnergy} be fulfilled and additionally suppose
    \begin{equation}
        \label{lemmaUpperBoundForTheGradientOfTheAdvectionHyperdiffusionEquationDecayAssumptionOnU}
        \|\nabla u\|_\infty = c_{\text{\tiny{$\nabla u$}}} (1+t)^{-\nu}.
    \end{equation}
    Then there exists a constant $c>0$ depending on $c_{\text{\tiny{$\nabla u$}}}$, $n$, $\alpha$, $\kappa$, $\|\theta_0\|_1$, $\|\theta_0\|_2$ and $\|\nabla \theta_0\|_2$ such that
    \begin{equation*}
        \begin{aligned}
            \|\nabla \theta\|_2 &\leq c (1+t)^{-\frac{n}{8}-\frac{1}{4}} e^{\frac{c_{\text{\tiny{$\nabla u$}}}}{1-\nu}\left((1+t)^{1-\nu}-1\right)} & &\textnormal{ if } 0<\nu <1,
            \\
            \|\nabla \theta\|_2 &\leq c (1+t)^{-\frac{n}{8}-\frac{1}{4}}  & &\textnormal{ if } \nu = 1,
            \\
            \|\nabla \theta\|_2 &\leq c (1+t)^{-\frac{n}{8}-\frac{1}{4}-c_{\text{\tiny{$\nabla u$}}}} e^{\frac{c_{\text{\tiny{$\nabla u$}}}}{\nu-1}\left(1-(1+t)^{-(\nu-1)}\right)} & &\textnormal{ if } \nu > 1.
        \end{aligned}
    \end{equation*}
    \end{lemma}
    \begin{proof}
    Differentiating \eqref{advectionHyperdiffusionEquation} yields
    \begin{equation*}
        \partial_t \nabla \theta + \kappa\nabla\Delta^2 \theta + \nabla (u\cdot\nabla \theta) =0
    \end{equation*}
    such that testing with $\nabla \theta$ and partial integration implies
    \begin{equation*}
        \begin{aligned}
            \frac{1}{2}\frac{d}{dt}\|\nabla\theta\|_2^2
            &= -\kappa \int_{\mathbb{R}^n}\nabla \theta \cdot \nabla\Delta^2 \theta \ dx - \int_{\mathbb{R}^n}\nabla\theta \cdot\nabla (u\cdot\nabla \theta) \ dx
            \\
            &= -\kappa \|\Delta\nabla\theta\|_2^2 - \sum_{i,j}\int_{\mathbb{R}^n}\partial_i \theta \partial_i u_j \partial_j \theta \ dx - \sum_{i,j} \int_{\mathbb{R}^n}\partial_i \theta u_j \partial_i\partial_j \theta \ dx.
        \end{aligned}
    \end{equation*}
    By partial integration and the impressibility condition on $u$ the third term on the right-hand side vanishes as
    \begin{equation*}
        \begin{aligned}
            \sum_{i,j}\int_{\mathbb{R}^n}\partial_i \theta u_j \partial_i\partial_j \theta \ dx 
            &= -\sum_{i,j}\int_{\mathbb{R}^n}\partial_j\partial_i \theta u_j \partial_i \theta \ dx - \sum_{i,j}\int_{\mathbb{R}^n}\partial_i \theta\partial_j u_j \partial_i \theta \ dx
            \\
            &= - \sum_{i,j}\int_{\mathbb{R}^n}\partial_i \theta u_j \partial_i\partial_j \theta \ dx.
        \end{aligned}
    \end{equation*}
    Therefore Hölder's inequality implies
    \begin{equation} 
        \label{lemmaUpperBoundForTheGradientOfTheAdvectionHyperdiffusionEquationEqA}
        \begin{aligned}
            \frac{1}{2}\frac{d}{dt}\|\nabla\theta\|_2^2
            &= -\kappa \|\Delta\nabla\theta\|_2^2 - \sum_{i,j}\int_{\mathbb{R}^n}\partial_i \theta \partial_i u_j \partial_j \theta \ dx
            \\
            &\leq -\kappa \|\Delta\nabla\theta\|_2^2 + \|\nabla u\|_\infty \|\nabla\theta\|_2^2.            
        \end{aligned}
    \end{equation}
    Similar to before define
    \begin{equation*}
        S(t) = \left\lbrace \xi\in \mathbb{R}^n \ \middle |\  |\xi| \leq r(t) \right \rbrace,
    \end{equation*}
    where $r=\left(\frac{\beta}{2\kappa(1+t)}\right)^\frac{1}{4}$ with a constant $\beta>0$. Then Plancherel and Lemma \ref{lemmaUpperBoundAdvectionHyperdiffusion} imply
    \begin{equation}
        \label{lemmaUpperBoundForTheGradientOfTheAdvectionHyperdiffusionEquationEqB}
        \begin{aligned}
            \|\Delta \nabla \theta\|_2^2 &= \int_{\mathbb{R}^n}|\xi|^6 |\hat\theta|^2 d\xi
            \geq \int_{S^c}  |\xi|^6 |\hat\theta|^2 d\xi
            \geq r^4 \int_{S^c}  |\xi|^2 |\hat\theta|^2 d\xi
            \\
            &= r^4 \left(\|\nabla \theta\|_2^2 - \int_{S} |\xi|^2 |\hat\theta|^2 d\xi \right)
            \geq r^4 \left(\|\nabla \theta\|_2^2 - r^2\int_{S}  |\hat\theta|^2 d\xi \right)
            \\
            &\geq r^4 \left(\|\nabla \theta\|_2^2 - r^2 \|\theta\|_2^2 \right) \geq r^4 \left(\|\nabla \theta\|_2^2 - cr^2 (1+t)^{-\frac{n}{4}} \right).
        \end{aligned}
    \end{equation}
    Combining \eqref{lemmaUpperBoundForTheGradientOfTheAdvectionHyperdiffusionEquationEqA} and \eqref{lemmaUpperBoundForTheGradientOfTheAdvectionHyperdiffusionEquationEqB}, the definition of $r$ yields
    \begin{equation*}
        \begin{aligned}
            \frac{d}{dt}\|\nabla\theta\|_2^2 &\leq -2\kappa r^4 \left(\|\nabla \theta\|_2^2 - cr^2 (1+t)^{-\frac{n}{4}} \right)+2 \|\nabla u\|_\infty \|\nabla\theta\|_2^2 
            \\
            &= 2\left(\|\nabla u \|_\infty - \beta (1+t)^{-1}\right) \|\nabla\theta\|_2^2 + c \beta^\frac{3}{2} (1+t)^{-\frac{n}{4}-\frac{3}{2}}.
        \end{aligned}
    \end{equation*}
    By Grönwall's inequality and the decay assumption \eqref{lemmaUpperBoundForTheGradientOfTheAdvectionHyperdiffusionEquationDecayAssumptionOnU}
    \begin{equation}
        \label{lemmaUpperBoundForTheGradientOfTheAdvectionHyperdiffusionEquationEqC}
        \begin{aligned}
            \|\nabla\theta\|_2^2 &\leq e^{2\int_0^t \|\nabla u \|_\infty - \beta (1+\tau)^{-1} \ d\tau} \|\nabla \theta_0\|_2^2 
            \\
            &\qquad+  c \beta^\frac{3}{2}\int_0^t e^{2\int_\tau^t \|\nabla u \|_\infty - \beta (1+s)^{-1} \ ds}  (1+\tau)^{-\frac{n}{4}-\frac{3}{2}}\ d\tau
            \\
            &\leq e^{2\int_0^t  c_{\text{\tiny{$\nabla u$}}}(1+\tau)^{-\nu} - \beta (1+\tau)^{-1} \ d\tau} \|\nabla \theta_0\|_2^2 
            \\
            &\qquad+  c \beta^\frac{3}{2}\int_0^t e^{2\int_\tau^t c_{\text{\tiny{$\nabla u$}}}(1+s)^{-\nu} - \beta (1+s)^{-1} \ ds}  (1+\tau)^{-\frac{n}{4}-\frac{3}{2}}\ d\tau,
        \end{aligned}
    \end{equation}
    where the second exponentiated integral can be calculated by
    \begin{equation*}
        \begin{aligned}
            e^{2\int_\tau^t  c_{\text{\tiny{$\nabla u$}}}(1+s)^{-\nu} - \beta (1+s)^{-1} \ ds} &= 
            \begin{cases}
                e^{2\left(\frac{c_{\text{\tiny{$\nabla u$}}}}{1-\nu} \left((1+t)^{1-\nu}-(1+\tau)^{1-\nu}\right) - \beta \left(\ln(1+t)-\ln(1+\tau)\right)\right)}
                \\
                e^{2(c_{\text{\tiny{$\nabla u$}}}-\beta)\left(\ln (1+t)-\ln(1+\tau)\right)}
                \\
                e^{2\left(\frac{c_{\text{\tiny{$\nabla u$}}}}{\nu-1} \left((1+\tau)^{-(\nu-1)}-(1+t)^{-(\nu-1)}\right) - \beta \left(\ln(1+t)-\ln(1+\tau)\right)\right)}
            \end{cases}
            \\
            &= \left(\frac{1+\tau}{1+t}\right)^{2\beta}
            \begin{cases}
                e^{2\frac{c_{\text{\tiny{$\nabla u$}}}}{1-\nu} \left((1+t)^{1-\nu}-(1+\tau)^{1-\nu}\right)} & \textnormal{ if } \nu< 1
                \\
                \left(\frac{1+t}{1+\tau}\right)^{2c_{\text{\tiny{$\nabla u$}}}} & \textnormal{ if } \nu= 1
                \\
                e^{2\frac{c_{\text{\tiny{$\nabla u$}}}}{\nu-1} \left((1+\tau)^{-(\nu-1)}-(1+t)^{-(\nu-1)}\right)} & \textnormal{ if } \nu> 1
            \end{cases}
        \end{aligned}
    \end{equation*}
    Choosing $\beta > \frac{1}{4} + \frac{n}{8} + c_{\text{\tiny{$\nabla u$}}}$ we can calculate the whole expression on the right-hand side of \eqref{lemmaUpperBoundForTheGradientOfTheAdvectionHyperdiffusionEquationEqB}. In case $0<\nu < 1$
    \begin{equation*}
        \begin{aligned}
            \|\nabla \theta\|_2^2 &\leq  (1+t)^{-2\beta} e^{2\frac{c_{\text{\tiny{$\nabla u$}}}}{1-\nu} \left((1+t)^{1-\nu}-1\right)}\|\nabla\theta_0\|_2^2 \\
            &\qquad + c \beta^\frac{3}{2} (1+t)^{-2\beta} \int_0^t (1+\tau)^{2\beta-\frac{n}{4}-\frac{3}{2}} e^{2\frac{c_{\text{\tiny{$\nabla u$}}}}{1-\nu} \left((1+t)^{1-\nu}-(1+\tau)^{1-\nu}\right)} \ d\tau
            \\
            &\leq  (1+t)^{-2\beta} e^{2\frac{c_{\text{\tiny{$\nabla u$}}}}{1-\nu} \left((1+t)^{1-\nu}-1\right)}\|\nabla\theta_0\|_2^2 \\
            &\qquad + c \beta^\frac{3}{2} (1+t)^{-2\beta}e^{2\frac{c_{\text{\tiny{$\nabla u$}}}}{1-\nu} \left((1+t)^{1-\nu}-1\right)} \int_0^t (1+\tau)^{2\beta-\frac{n}{4}-\frac{3}{2}}  \ d\tau
            \\
            &=  (1+t)^{-2\beta} e^{2\frac{c_{\text{\tiny{$\nabla u$}}}}{1-\nu} \left((1+t)^{1-\nu}-1\right)}\|\nabla\theta_0\|_2^2 \\
            &\qquad +  \frac{c \beta^\frac{3}{2}}{2\beta -\frac{n}{4}-\frac{1}{2}} (1+t)^{-2\beta}e^{2\frac{c_{\text{\tiny{$\nabla u$}}}}{1-\nu} \left((1+t)^{1-\nu}-1\right)} \left((1+\tau)^{2\beta-\frac{n}{4}-\frac{1}{2}} -1 \right)
            \\
            &\leq c  (1+t)^{-\frac{n}{4}-\frac{1}{2}} e^{2\frac{c_{\text{\tiny{$\nabla u$}}}}{1-\nu} \left((1+t)^{1-\nu}-1\right)}.
        \end{aligned}
    \end{equation*}
    In case $\nu=1$
    \begin{equation*}
        \begin{aligned}
            \|\nabla\theta\|_2^2
            &\leq (1+t)^{2(c_{\text{\tiny{$\nabla u$}}}-\beta)} \|\nabla \theta_0\|_2^2 +  c \beta^\frac{3}{2}(1+t)^{2(c_{\text{\tiny{$\nabla u$}}}-\beta)}\int_0^t (1+\tau)^{2\beta- 2c_{\text{\tiny{$\nabla u$}}}-\frac{n}{4}-\frac{3}{2}}\ d\tau
            \\
            &= (1+t)^{2(c_{\text{\tiny{$\nabla u$}}}-\beta)} \|\nabla \theta_0\|_2^2
            \\
            &\qquad +  \frac{c \beta^\frac{3}{2}}{2\beta -2 c_{\text{\tiny{$\nabla u$}}}-\frac{n}{4}-\frac{1}{2}}(1+t)^{2(c_{\text{\tiny{$\nabla u$}}}-\beta)}\left((1+t)^{2\beta -2 c_{\text{\tiny{$\nabla u$}}}-\frac{n}{4}-\frac{1}{2}}-1\right)
            \\
            &\leq c (1+t)^{-\frac{n}{4}-\frac{1}{2}}.
        \end{aligned}
    \end{equation*}
    In case $\nu>1$ 
    \begin{equation*}
        \begin{aligned}
            \|\nabla\theta\|_2^2 
            &\leq (1+t)^{-2\beta} e^{2\frac{c_{\text{\tiny{$\nabla u$}}}}{\nu-1} \left(1-(1+t)^{-(\nu-1)}\right)} \|\nabla \theta_0\|_2^2 
            \\
            &\qquad+  c \beta^\frac{3}{2}(1+t)^{-2\beta} e^{2\frac{c_{\text{\tiny{$\nabla u$}}}}{\nu-1} \left(1-(1+t)^{-(\nu-1)}\right)}
            \\
            &\qquad\qquad \cdot \int_0^t e^{2\frac{c_{\text{\tiny{$\nabla u$}}}}{\nu-1} \left((1+\tau)^{-(\nu-1)}-1\right)} (1+\tau)^{2\beta-\frac{n}{4}-\frac{3}{2}}\ d\tau
        \end{aligned}
    \end{equation*}
    \begin{equation*}
        \begin{aligned}
            &\lesssim (1+t)^{-2\beta} e^{2\frac{c_{\text{\tiny{$\nabla u$}}}}{\nu-1} \left(1-(1+t)^{-(\nu-1)}\right)} \|\nabla \theta_0\|_2^2 
            \\
            &\qquad+  c \beta^\frac{3}{2}(1+t)^{-2\beta} e^{2\frac{c_{\text{\tiny{$\nabla u$}}}}{\nu-1} \left(1-(1+t)^{-(\nu-1)}\right)} \int_0^t (1+\tau)^{2\beta-2c_{\text{\tiny{$\nabla u$}}}-\frac{n}{4}-\frac{3}{2}}\ d\tau
            \\
            &=(1+t)^{-2\beta} e^{2\frac{c_{\text{\tiny{$\nabla u$}}}}{\nu-1} \left(1-(1+t)^{-(\nu-1)}\right)} \|\nabla \theta_0\|_2^2 
            \\
            &\qquad+  \frac{c \beta^\frac{3}{2}}{2\beta-2c_{\text{\tiny{$\nabla u$}}}-\frac{n}{4}-\frac{1}{2}} e^{2\frac{c_{\text{\tiny{$\nabla u$}}}}{\nu-1} \left(1-(1+t)^{-(\nu-1)}\right)}
            \\
            &\qquad\qquad \cdot (1+t)^{-2\beta} \left( (1+\tau)^{2\beta-2c_{\text{\tiny{$\nabla u$}}}-\frac{n}{4}-\frac{1}{2}} -1\right)
            \\
            &\leq c e^{2\frac{c_{\text{\tiny{$\nabla u$}}}}{\nu-1} \left(1-(1+t)^{-(\nu-1)}\right)}(1+t)^{-2 c_{\text{\tiny{$\nabla u$}}}-\frac{n}{4}-\frac{1}{2}}.
        \end{aligned}
    \end{equation*}
    \end{proof}
    Combining Theorem \ref{theoremLowerBoundForTheEnergy} and Lemma \ref{lemmaUpperBoundForTheGradientOfTheAdvectionHyperdiffusionEquation} we are now able to prove the lower bounds for the mixing norms.
    \begin{proof}[Theorem \ref{theoremMixingBound}]
    By Plancherel and Hölder's inequality
    \begin{equation*}
        \|\theta\|_2^2 = \|\hat\theta\|_2^2 = \left\| |\xi| \hat\theta\  |\xi|^{-1} \hat\theta \right\|_1 \leq \left\| |\xi| \hat\theta \right\|_2 \left\| |\xi|^{-1} \hat\theta \right\|_2 = \|\nabla \theta\|_2 \|\nabla^{-1}\theta\|_2
    \end{equation*}
    such that using the lower bound on $\theta$ of Theorem \ref{theoremLowerBoundForTheEnergy} and the upper bound on $\nabla\theta$ of Lemma \ref{lemmaUpperBoundForTheGradientOfTheAdvectionHyperdiffusionEquation} one finds
    \begin{equation*}
        \begin{aligned}
            \|\nabla^{-1}\theta\|_2 &\geq \|\theta\|_2^2 \|\nabla\theta\|_2^{-1}
            \\
            &\gtrsim (1+t)^{-\frac{n}{8}+\frac{1}{4}}
            \begin{cases}
                e^{-\frac{c_{\text{\tiny{$\nabla u$}}}}{1-\nu}\left((1+t)^{1-\nu}-1\right)} & \textnormal{ if } 0<\nu<1, \\
                1 & \textnormal{ if } \nu = 1,\\
                (1+t)^{c_{\text{\tiny{$\nabla u$}}}}e^{-\frac{c_{\text{\tiny{$\nabla u$}}}}{\nu - 1}\left(1-(1+t)^{-(\nu-1)}\right)}& \textnormal{ if } \nu > 1\\
            \end{cases}
        \end{aligned}
    \end{equation*}
    and for the filamentation length
    \begin{equation*}
        \begin{aligned}
            \lambda &= \frac{\|\nabla^{-1}\theta\|_2}{\|\theta\|_2}\geq \|\theta\|_2\|\nabla\theta\|_2^{-1}
            \\
            &\gtrsim (1+t)^{\frac{1}{4}}
            \begin{cases}
                e^{-\frac{c_{\text{\tiny{$\nabla u$}}}}{1-\nu}\left((1+t)^{1-\nu}-1\right)} & \textnormal{ if } 0<\nu<1, \\
                1 & \textnormal{ if } \nu = 1,\\
                (1+t)^{c_{\text{\tiny{$\nabla u$}}}}e^{-\frac{c_{\text{\tiny{$\nabla u$}}}}{\nu - 1}\left(1-(1+t)^{-(\nu-1)}\right)}& \textnormal{ if } \nu > 1.
            \end{cases}
        \end{aligned}
    \end{equation*}
    \end{proof}

\section{Appendix}
The following properties of the hyperdiffusion kernel can be found in \cite{gazzolaGrunauSomeNewPropertiesOfBiharmonicHeatKernels}.
\begin{proposition}
\label{propositionHyperdiffusionKernelCitedProporties}
The hyperdiffusion kernel is given by
\begin{equation*}
    G(t,x)= \alpha_n t^{-\frac{n}{4}} f_n \left(\frac{|x|}{t^\frac{1}{4}}\right),
\end{equation*}
where $\alpha_n$ is a normalization constant and
\begin{itemize}
    \item 
        \begin{equation*}
            f_n(\eta) = \eta^{1-n} \int_0^\infty e^{-s^4} (\eta s)^\frac{n}{2} J_{\frac{n-2}{2}} (\eta s)\ ds,
        \end{equation*}
        where $J_\nu$ denotes the $\nu$-tv Bessel function of the first kind.
    \item 
        there exists $K_n>0$ and $\mu_n>0$ such that
        \begin{equation*}
            |f_n(\eta)| \leq K_n e^{-\mu_n \eta^\frac{4}{3}}
        \end{equation*}
        for all $\eta\geq 0$.
    \item
        \begin{equation*}
            f_n' (\eta) = - \eta f_{n+2}(\eta)
        \end{equation*}
        for all $n\geq 1$.
\end{itemize}
\end{proposition}

\begin{lemma}[Bounds for the hyperdiffusion equation]
\label{appendixLemmaBoundsOnHyperdiffusionEquation}
For $1\leq p< \infty$ and $t>0$ the hyperdiffusion kernel fulfills
\begin{equation*}
    \begin{aligned}
        \|G\|_p &\lesssim_{n,p} t^{-\frac{n}{4}\left(1-\frac{1}{p}\right)}\\
        \|\nabla G\|_p &\lesssim_{n,p} t^{-\frac{n}{4}\left(1-\frac{1}{p}\right)-\frac{1}{4}}
    \end{aligned}
\end{equation*}
and the solution of the hyperdiffusion equation \eqref{lemmaLowerBoundHyperdiffusionEquationPDEforT} can be bounded by
\begin{equation*}
    \begin{aligned}
        \|T\|_2 &\lesssim (\kappa t)^{-\frac{n}{8}}\|\theta_0\|_1\\
        \|\nabla T\|_2 &\lesssim (\kappa t)^{-\frac{n}{8}-\frac{1}{4}}\|\theta_0\|_1\\
        \|\nabla T\|_\infty &\lesssim (\kappa t)^{-\frac{n}{8}-\frac{1}{4}}\|\theta_0\|_2\\
    \end{aligned}
\end{equation*}
and
\begin{equation*}
    \left\|e^{-\kappa |\xi|^4t} \hat \theta\right\|_2 \lesssim (\kappa t)^{-\frac{n}{8}} \|\theta_0\|_1.
\end{equation*}
\end{lemma}
\begin{proof}
By passing to spherical coordinates, the change of variables $\rho = r t^{-\frac{1}{4}}$ and Proposition \ref{propositionHyperdiffusionKernelCitedProporties} one gets
\begin{equation*}
    \begin{aligned}
        \|G\|_p^p &= a_n^p t^{-\frac{pn}{4}} \int_{\mathbb{R}^n}\left|f_n \left(\frac{|x|}{t^\frac{1}{4}}\right)\right|^p dx \sim t^{-\frac{pn}{4}} \int_0^\infty r^{n-1} \left|f_n \left(\frac{r}{t^\frac{1}{4}}\right)\right|^p \ dr 
        \\
        &= t^{-\frac{n}{4}(p-1)} \int_0^\infty \rho^{n-1} \left|f_n \left(\rho\right)\right|^p \ d\rho \lesssim t^{-\frac{n}{4}(p-1)} \int_0^\infty \rho^{n-1} e^{-p\mu_n \rho^\frac{4}{3}} \ d\rho \lesssim t^{-\frac{n}{4}(p-1)}
    \end{aligned}
\end{equation*}
and similarly
\begin{equation*}
    \begin{aligned}
        \|\nabla G\|_p^p &= a_n^p t^{-\frac{pn}{4}} \int_{\mathbb{R}^n}\left|\nabla f_n \left(\frac{|x|}{t^\frac{1}{4}}\right)\right|^p dx \sim t^{-\frac{p}{4}(n+1)} \int_{\mathbb{R}^n}\left| f_n' \left(\frac{|x|}{t^\frac{1}{4}}\right)\right|^p dx 
        \\
        &\sim t^{-\frac{p}{4}(n+1)} \int_0^\infty r^{n-1} \left| f_n' \left(\frac{r}{t^\frac{1}{4}}\right)\right|^p dr \sim t^{-\frac{n}{4}\left(p-1\right)-\frac{p}{4}} \int_0^\infty \rho^{n-1} |f_n'(\rho)|^p\ d\rho 
        \\
        &= t^{-\frac{n}{4}\left(p-1\right)-\frac{p}{4}} \int_0^\infty \rho^{n} |f_{n+2}(\rho)|^p\ d\rho \lesssim t^{-\frac{n}{4}\left(p-1\right)-\frac{p}{4}} \int_0^\infty \rho^{n} e^{-p\mu_{n+2}\eta^\frac{4}{3}}\ d\rho 
        \\
        &\lesssim t^{-\frac{n}{4}\left(p-1\right)-\frac{p}{4}}.
    \end{aligned}
\end{equation*}
The bounds on $T$ follow directly as
\begin{equation*}
    T(t,\cdot) = G(\kappa t, \cdot)\star T_0(\cdot),
\end{equation*}
Young's inequality for convolution and the previous bounds on $G$ yield
\begin{equation*}
    \begin{aligned}
        \|T\|_2 &\leq \|G\star \theta_0\|_2 \leq \|G\|_2 \|\theta_0\|_1 \lesssim (\kappa t)^{-\frac{n}{8}} \|\theta_0\|_1\\
        \|\nabla T\|_2 &\leq \|\nabla G\star \theta_0\|_2 \leq \|\nabla G\|_2 \|\theta_0\|_1 \lesssim (\kappa t)^{-\frac{n}{8}-\frac{1}{4}} \|\theta_0\|_1\\
        \|\nabla T\|_\infty &\leq \|\nabla G\star \theta_0\|_\infty \leq \|\nabla G\|_2 \|\theta_0\|_2 \lesssim (\kappa t)^{-\frac{n}{8}-\frac{1}{4}} \|\theta_0\|_2.
    \end{aligned}
\end{equation*}
Analogously by Plancherel
\begin{equation*}
    \left\|e^{-\kappa |\xi|^4t} \hat \theta_0\right\|_2 = \left\|G(\kappa t,\cdot )\star \theta_0(\cdot)\right\|_2 \leq \|G\|_2\|\theta_0\|_1 \lesssim (\kappa t)^{\frac{n}{8}}\|\theta\|_1.
\end{equation*}
\end{proof}

\section*{Acknowledgment}
FB acknowledges the support by the Deutsche Forschungsgemeinschaft (DFG) within the Research Training Group GRK 2583 "Modeling, Simulation and Optimization of Fluid Dynamic Applications". CN was partially supported by DFG-TRR181 and GRK-2583.

 
\bibliographystyle{plain}
\bibliography{m1_literature}

\end{document}